\documentclass[12pt]{article}
\usepackage[latin1]{inputenc}
\usepackage[all]{xy}
\usepackage{latexsym}
\usepackage{amssymb}
\usepackage{amsmath}
\usepackage{mathrsfs}
\usepackage{graphicx}
\usepackage{verbatim}

\newcounter{Num}
\newtheorem{theo}{Theorem}[section]

\newtheorem{lemma}[theo]{Lemma}
\newtheorem{corollary}[theo]{Corollary}
\newtheorem{prop}[theo]{Proposition}

\newenvironment{proof}{ \textbf{Proof.}}{$\Box$}

\newcommand {\ZZ} {\mathbb {Z}}
\newcommand {\romanV} {${\overline{\underline{\mathrm{V}}}}$}

\newcommand {\CC} {\mathbb {C}}

\renewcommand{\ss}{\mathfrak{s}}
\newcommand{\kk}{\mathfrak{k}}

\newcommand{\mm}{\mathfrak{m}}

\newcommand{\hh}{\mathfrak{h}}
\newcommand{\pp}{\mathfrak{p}}
\renewcommand{\gg}{\mathfrak{g}}
\renewcommand{\ll}{\mathfrak{l}}
\newcommand{\bb}{\mathfrak{b}}

\renewcommand {\phi} {\varphi}

\newcommand{\refth}[1]{Theorem \ref{#1}}
\newcommand{\refle}[1]{Lemma \ref{#1}}
\newcommand{\refsec}[1]{Section \ref{#1}}
\newcommand{\refcor}[1]{Corollary \ref{#1}}
\newcommand{\refprop}[1]{Proposition \ref{#1}}
\newcommand{\refeq}[1]{(\ref{#1})}

\renewcommand{\Im}{\mathrm{Im}}

\renewcommand{\sp}{{sp}}

\newcommand{\gl}{{g\ell}}
\renewcommand{\sl}{{s\ell}}
\newcommand{\so}{{so}}

\newcommand{\ind}{\mathrm{ind}}
\newcommand{\qed}{$\Box$}
\newcommand{\Gr}{\mathrm{Gr}}

\newcommand{\eps}{\varepsilon}

\def\cplus{\hbox{$\subset${\raise0.3ex\hbox{\kern -0.55em ${\scriptscriptstyle +}$}}\ }}

\def\clplus{\hbox{$\subset${\raise0.3ex\hbox{\kern -0.55em ${\scriptscriptstyle +}$}}\ }}
\def\crplus{\hbox{$\supset${\raise1.05pt\hbox{\kern -0.72em ${\scriptscriptstyle +}$}}\ }}

\def\p{\bold p}

\def\ind{\text{\rm ind}}

\def\Im{\text{\rm Im}}

\def\p{\bold p}

\def\Hom{\text{\rm Hom}}
\def\m{\frak m}
\def\n{\frak n}
\def\p{\frak p}
\def\g{\frak g}

\def\b{\frak b}
\def\k{\frak k}
\def\t{\frak t}

\def\r{\frak r}
\def\rr{\frak r}

\title{ \begin{flushright}
\begin{small}
To Bertram Kostant\end{small}
\bigskip
\end{flushright}
A CONSTRUCTION OF GENERALIZED HARISH-CHANDRA MODULES FOR LOCALLY REDUCTIVE LIE ALGEBRAS}
\author{ Ivan Penkov and Gregg Zuckerman}
\date{}
\begin{document}

\maketitle \abstract We study cohomological induction for a pair $(\frak g,\frak k)$, $\frak g$ being an infinite dimensional locally reductive Lie algebra and $\frak k \subset\frak g$ being of the form $\frak k_0 + C_\gg(\frak k_0)$, where $\frak k_0\subset\frak g$ is a finite dimensional reductive in $\frak g$ subalgebra and $C_{\gg} (\frak k_0)$ is the centralizer of $\frak k_0$ in $\frak g$.  We prove a general non-vanishing and $\frak k$-finiteness theorem for the output.  This yields in particular simple $(\frak g,\frak k)$-modules of finite type over $\frak k$ which are analogs of the fundamental series of generalized Harish-Chandra modules constructed in \cite{PZ1} and \cite{PZ2}.  We study explicit versions of the construction when $\frak g$ is a root-reductive or diagonal locally simple Lie algebra.
\endabstract
{ (2000 MSC): Primary 17B10, Secondary 17B55 }

\section{Introduction}

A \emph{locally reductive Lie algebra} is defined as a union $\cup_{n\in\ZZ_{> 0}}\frak g_n$ of nested finite dimensional reductive Lie algebras $\frak g_n \subset \frak g_{n+1}$ such that each $\frak g_n$ is reductive in $\frak g_{n+1}$.  The class of locally reductive Lie algebras is a very natural and interesting class of infinite dimensional Lie algebras, and no classification is known.  There are two (intersecting) subclasses of locally reductive Lie algebras which are relatively well-understood, see Subsection \ref{sec23}: the root-reductive Lie algebras, \cite{DP}, \cite{B}, and the locally simple diagonal Lie algebras, \cite{BZh}.  For instance, the Lie algebra $\gl (\infty)$ of infinite matrices with only finitely many non-zero entries is root-reductive, and the Lie algebra $\gl (2^\infty)$, defined as the union $\cup_{n\in\ZZ_{> 0}} \gl(2^n)$ via the injections
\begin{eqnarray*}
\gl (2^n)  &\subset & \gl (2^{n+1}) \\
A &\longmapsto &\begin{pmatrix} A & 0 \\ 0 & A\end{pmatrix},
\end{eqnarray*}
is diagonal. Both of the above classes of Lie algebras yield explicit examples of the general construction of this paper.

Representations of direct limit Lie groups have been studied for quite a considerable time now, \cite{Ha}, \cite{Ne}, \cite{O1}, \cite{O2}, \cite{NO}, \cite{W}, \cite{NRW}, however the theory of direct limit group representations has not been related in a systematic way to modules over the direct limit Lie algebra. In our opinion, this problem deserves further investigation.

In this paper we restrict ourselves to representations of locally reductive Lie algebras $\gg$ and we initiate the study of $(\frak g,\frak k)$-modules of finite type over $\frak k$. More specifically, we provide a construction of such modules when $\frak k$ is the form $\frak k_0 + C_{\frak g}(\frak k_0)$ for a finite-dimensional reductive in $\frak g$ subalgebra $\frak k_0$ ($C_{\frak g}(\cdot)$ denotes centralizer in $\frak g$).  If $\frak g$ is root-reductive, such subalgebras $\frak k$ may equal the fixed vectors of an involution on $\frak g$, hence $(\frak g,\frak k)$-modules of finite type generalize Harish-Chandra modules. Our main construction is a generalization of the fundamental series for subalgebras $\frak k\subset\frak g$ of the form $\frak k = \frak k_0 + C_{\frak g} (\frak k_0)$, cf. \cite{PZ2}.  We use the derived functor of the functor of locally finite $\frak k_0$-vectors. Its output is automatically endowed with a $(\frak g,\frak k)$-module structure.  Our finiteness result is based on a general finiteness theorem for cohomological induction which asserts $\frak k$-finiteness of the output provided the input is $\frak k\cap\frak m$-finite, $\frak m$ being the reductive part of the compatible parabolic subalgebra.  A main technical observation of this paper is that one can construct reasonably large classes of parabolically induced modules which are $\frak k\cap\frak m$-finite, both when $\frak g$ is root-reductive and when $\frak g$ is a diagonal. This is based on the stabilization of the branching multiplicities of certain tensor representations of classical Lie algebras of increasing rank.

Our main interest is in constructing simple $(\frak g,\frak k)$-modules $M$ which in addition to being of finite type are also strict, i.e. for which $\frak k$ coincides with the subalgebra of $\frak g$ consisting of all elements $g\in\frak g$ which act locally finitely on $M$ (the Fernando-Kac subalgebra of $M$). In particular, we provide sufficient conditions for strictness of the modules constructed.

The theory of $(\frak g, \frak k)$-modules for locally reductive Lie algebras $\frak g$ is still in its infancy and many questions remain off limits for this paper.  This concerns for instance the problem of unitarizability of the $(\frak g, \frak k)$-modules we construct.  Another very interesting problem is to describe the locally reductive subalgebras $\frak k\subset\frak g$ which admit strict simple $(\frak g, \frak k)$-modules of finite type.  Our paper deals with subalgebras of the form $\frak k_0 + C_{\frak g} (\frak k_0)$, and hence not with the case when $\frak k = \frak h$ is a splitting Cartan subalgebra of $s\ell (\infty), so(\infty)$ and $sp(\infty)$.  In fact, strict simple $(\frak g,\frak h)$-modules of finite type exist only for $s\ell (\infty)$ and $sp(\infty)$, and for $s\ell (\infty)$, and I. Dimitrov has been working on their classification, \cite{Di}.  Finally, we would like to point out that the idea of studying direct limits of cohomologically induced modules was first suggested by A. Habib in \cite{Ha} and that this idea has been an inspiration for us.

{\bf Acknowledgement.}  We thank G. Olshanskii for helpful comments and J. Willenbring for a detailed discussion of \refprop{prop11}. G. Zuckerman acknowledges the hospitality and support of the Jacobs University Bremen.

\section{Preliminaries}\label{sec2}

\subsection{Conventions}  All vector spaces and Lie algebras are defined over $\CC$.  If $p$ is a positive integer and $W$ is a vector space or a Lie algebra, we set $W^p := \underbrace{W\oplus\ldots\oplus W}_{ p~{\text{times}}}$. $T^\cdot(W)=\bigoplus_{k\geq 0}T^k(W)$ is the tensor algebra of $W$. The superscript $^*$ indicates dual space, and $\otimes=\otimes_\CC$. If $\frak g$ is a Lie algebra, $Z_{\frak g}$ stands for the center of $\frak g$, $C_{\frak g} (\alpha)$ stands for the centralizer in $\frak g$ of a subset $\alpha\subset\frak g$, $U(\frak g)$ stands for the enveloping algebra and $Z_{U(\frak g)}$ stands for the center of $U(\frak g)$.  The sign $\cplus$ denotes semidirect sum of Lie algebras.  A subalgebra $\frak k\subset\frak g$ is {\it reductive in} $\frak g$ if under the adjoint action of $\frak k, \frak g$ is a semisimple $\frak k$-module.  If $\frak l$ is any subalgebra of $\frak g$ and $M$ is an $\frak l$-module, we denote the induced module $U(\frak g)\otimes_{U(\ll)} M$ by $\text{ind}^\frak g_\frak l M$.  If $\frak l'$ is a finite dimensional Lie algebra, by $V_{\frak l'}(\lambda)$ we denote the simple finite dimensional $\frak l'$-module with highest weight $\lambda$.  When we write a vector space $W$ as $\cup_{n\in\ZZ_{> 0}} W_n$ we automatically assume that $W_n\subset W_{n+1}$ for $n\in\ZZ_{> 0}$.

\subsection{A stabilization result}
\begin{prop}\label{prop11}  Let $\frak s_n$ be a sequence of classical
finite dimensional simple Lie algebras of rank $n$ and of fixed type $A, B, C$ or $D$.  Denote by $V_n$ the natural $\frak s_n$-module. Then, for any fixed $a,b,c,k\in\ZZ_{> 0}$ the length of the $\frak s_n$-module $T^k(V^a_n \oplus (V^*_n)^b\oplus \CC^c)$ stabilizes when $n\to\infty$ (here $\CC$ stands for the trivial $1$-dimensional $\frak s_n$-module).
\end{prop}
\begin{proof}
This result is a relatively straightforward corollary of the results in \cite{HTW}, and we describe the argument only very briefly. Assume that $\ss_n=\sl(n+1)$, let $\hh_n$ be the diagonal subalgebra, $\bb_n$ be the upper-triangular subalgebra, and $\eps_1-\eps_2,\dots, \eps_n-\eps_{n+1}$ be the standard basis in $\hh_n^*$. We will view any $\bb_n$-dominant weight $\lambda=\sum_{i=1}^{n+1}\lambda_i\eps_i$ of $\ss_n$, $\lambda_1\geq\dots\geq\lambda_n$, $\lambda_i\in\ZZ$ as a $\bb_{n+k}$-dominant weight of $\ss_{n+k}$ by inserting $k$ zeroes in the non-increasing sequence $\lambda_1\geq\dots\geq\lambda_{n+1}$ so that the remaining sequence remains non-increasing. Therefore, for a fixed $n_0$ and a $\bb_{n_0}$-dominant weight $\lambda$ as above, the $\ss_n$-module $V_{\ss_n}(\lambda)$ is well defined for $n\geq n_0$. The first fact needed in the proof of \refprop{prop11} is that for fixed $a,b,c,k$, there is an integer $n_0$ such that all simple constituents of $X_n:=T^k(V_n^a\oplus(V_n^*)^b\oplus\CC^c)$ are of the form $V_{\ss_n}(\lambda)$ for $n\geq n_0$, where $\lambda$ runs over a finite set of $\bb_{n_0}$-dominant weights of $\ss_{n_0}$. This is proved by a straightforward induction on $k$.

All that remains to show now is that for each $V_{\ss_n}(\lambda)$ with $\lambda$ as above, $\dim \Hom_{\ss_n}(V_{\ss_n}(\lambda), X_n)$ stabilizes when $n\to\infty$. This can also be done by induction on $k$. The case $k=1$ is obvious, so we can assume that the statement is true for $1,2,\dots,k$. Then, in order to prove the Proposition for $k+1$, it suffices to show that $\dim \Hom_{\ss_n}(V_{\ss_n}(\lambda), X_n\otimes V_n)$ and $\dim \Hom_{\ss_n}(V_{\ss_n}(\lambda), X_n\otimes V_n^*)$ stabilize for $n\to \infty$. Note that
\[
\dim \Hom_{\ss_n}(V_{\ss_n}(\lambda), X_n\otimes V_n)= \dim \Hom_{\ss_n}(V_{\ss_n}(\lambda)\otimes V_n^*, X_n),
\]
\[
\dim \Hom_{\ss_n}(V_{\ss_n}(\lambda), X_n\otimes V_n^*)= \dim \Hom_{\ss_n}(V_{\ss_n}(\lambda)\otimes V_n, X_n).
\]
The statement follows now from the induction assumption and from the key formula 1.2.1 in \cite{HTW} which implies that
\[
\Hom_{\ss_n}(V_{\ss_n}(\lambda'), V_{\ss_n}(\lambda)\otimes V_n^*)\neq 0
\]
for an independent on $n$ finite set of weights $\lambda'$ only (respectively,
\[
\Hom_{\ss_n}(V_{\ss_n}(\lambda''), V_{\ss_n}(\lambda)\otimes V_n)\neq 0
\]
for an independent on $n$ finite set $\lambda''$ only), and that $\dim \Hom_{\ss_n}(V_{\ss_n}(\lambda'), V_{\ss_n}(\lambda)\otimes V_n^*)$ (resp., $\dim\Hom_{\ss_n}(V_{\ss_n}(\lambda''), V_{\ss_n}(\lambda)\otimes V_n)$) stabilizes for $n\to\infty$. The reader will easily fill in the details.

For $\ss_n$ of types $B,C,D$ the argument is essentially the same and uses formulas 1.2.2 and 1.2.3 in \cite{HTW}.
\end{proof}

\subsection{Locally reductive Lie algebras}\label{sec23}  We defined locally reductive Lie algebras in the Introduction.  In the rest of this paper, when writing $\frak g = \cup_{n\in\ZZ_{> 0}} \frak g_n$ for a locally reductive Lie algebra $\frak g$, we will always assume that the $\frak g_n$'s form a chain
\begin{equation}\label{eq1}
\frak g_1\subset\frak g_2\subset\ldots\subset \frak g_n \subset\frak g_{n+1} \subset\ldots
\end{equation}
of finite dimensional reductive Lie algebras such that each $\frak g_n$ is reductive in $\frak g_{n+1}$.

An important but quite restrictive class of locally reductive Lie algebras are the {\it root-reductive} Lie algebras.  They have the form $\cup_{n\in\ZZ_{> 0}}\frak g_n$, where the chain (1) satisfies the requirement that each inclusion $\frak g_n\subset\frak g_{n+1}$ is a root homomorphism, i.e. maps a Cartan subalgebra of $\frak g_n$ into a Cartan subalgebra of $\frak g_{n+1}$ and any root space of $\frak g_n$ into a root space of $\frak g_{n+1}$. A most natural example of a root-reductive Lie algebra is the Lie algebra $\gl(\infty)$, defined via the chain $\gl(i)\subset\gl(i+1)$ of upper left-hand corner embeddings.

A Lie algebra $\frak s$ is {\it locally simple} if $\frak s = \cup_{n\in\ZZ_{>0}}\frak s_n$ where $\frak s_n$ are simple Lie algebras (in this case $\frak s_n$ is automatically reductive in $\frak s_{n+1}$), in particular a locally simple Lie algebra is locally reductive. Up to isomorphism, there are three simple infinite dimensional locally simple root-reductive Lie algebras: $s\ell (\infty), so(\infty)$ and $sp(\infty)$. They are defined by obvious chains of inclusions which are root-homomorphisms (in the case of $so(\infty)$ there are two natural choices: $\ldots\subset so(2i)\subset so(2i+2)\subset\ldots$ and $\ldots\subset so(2i+1)\subset so(2i+3)\subset\ldots$, however these yield isomorphic locally simple Lie algebras).  The following structure theorem has been proved in \cite{DP}.

\begin{theo}
Let $\frak g$ be a root-reductive Lie algebra.

\item[(a)]  The exact sequence
$$
0\to [\frak g,\frak g] \to \frak g \to \frak a := \frak g /[\frak g,\frak g] \to 0
$$
splits, hence $\frak g$ is isomorphic to the semidirect sum $[\frak g,\frak g]\cplus\frak a$ ($\frak a$ being an abelian Lie algebra).

\item[(b)] $[\frak g,\frak g]$ is isomorphic to a direct sum of at most countably many copies of \newline $s\ell (\infty), so(\infty), sp(\infty)$, as well as of simple finite dimensional Lie algebras.
\end{theo}

A more general and very interesting class of locally reductive Lie algebras which are not necessarily root-reductive are the {\it diagonal} Lie algebras.  By definition, a chain (1) of classical finite dimensional Lie algebras is {\it diagonal}, if for any $n$, the natural representation of $\frak g_{n+1}$ is isomorphic to a direct sum of copies of the natural representation of $\frak g_n$, of its dual and of the trivial representation.  Locally simple diagonal Lie algebras have been classified up to isomorphism in \cite{BZh}.  In the present paper, we will restrict ourselves to the simplest subclass of diagonal Lie algebras $\gl (p\Theta)$ defined below, however our results should extend without significant difficulty to general diagonal Lie algebras.  Let $\theta_1,\theta_2,\ldots$ be an infinite sequence of integers greater than $1$.  We denote by $\Theta$ the formal product $\theta_1\theta_2\ldots$ and, for each $p\in\ZZ_{\geq 1}$, we define the Lie algebra $\gl (p\Theta)$ (for $p=1$ we write simply $\gl (\Theta))$ as the union of the following diagonal chain
$$
\gl (p)\subset \gl (p\theta_1)\subset \gl (p\theta_1\theta_2)\subset\ldots
$$
where, for $n\in\ZZ_{\geq 0}$, $\gl (p\theta_1\theta_2\ldots \theta_{n-1})$ is embedded into $\gl (p\theta_1\ldots\theta_{n})$ by repeating a matrix $A\in \gl (p \theta_1\ldots\theta_{n-1}) \,\,  \theta_{n}$ times along the main diagonal in $\gl (p \theta_1\ldots \theta_{n})$.  The locally simple diagonal Lie algebra $s\ell (p\Theta )$ is defined in the same way with $\gl (p \theta_1\ldots\theta_n)$ replaced by $s\ell (p \theta_1\ldots\theta_n)$. The reader will check immediately that $\gl (p\Theta ) = Z_{\gl(p\Theta)}\oplus s\ell (p\Theta)$, the center $Z_{\gl (p\Theta)}$ being $1$-dimensional.  The Lie algebra $\gl (2^\infty)$ (see the Introduction) is the simplest example of a Lie algebra of the form $g\ell (p\Theta)$ (here $p=2=\theta_n, n\in\ZZ_{> 0}$).

\subsection{$(\frak g,\frak k)$-modules}  If $\frak g$ is a locally reductive Lie algebra and $M$ is a $\frak g$-module, the {\it Fernando-Kac subalgebra} $\frak g [M]\subset\frak g$ consists of all elements $g\in\frak g$ which act locally finitely on $M$, see \cite{F}, \cite{DMP} and the references therein.

If $\frak g$ is locally reductive and $\frak k\subset\frak g$ is a Lie subalgebra, we call a $\frak g$-module $M$ a $(\frak g,\frak k)$-{\it module} if $\frak k\subset\frak g [M]$.  In other words, $M$ is a $(\frak g,\frak k)$-module if for any $m\in M$ and any $n\in\ZZ_{>0}$ the $\frak k_n$-submodule of $M$ generated by $m$ is finite-dimensional.  We call a $(\frak g,\frak k)$-module $M$ {\it strict} if $\frak k = \frak g [M]$. Sometimes we use the term $\frak k$-{\it integrable} $\frak g$-{\it module} as an equivalent to $(\frak g,\frak k)$-module.

Furthermore, we define a $(\frak g,\frak k)$-module $M$ to be of {\it finite type} if the following two conditions hold:

\item[-] every finitely generated $\frak k$-submodule $M'$ of $M$ has finite length as a $\frak k$-module;
\item[-] for every fixed simple integrable $\frak k$-module $L$, the multiplicity of $L$ as a subquotient of $M'$ is bounded when $M'$ runs over all finitely generated $\frak k$-submodules of $M$. If a $(\frak g,\frak k)$-module $M$ is not of {\it finite type}, we say that $M$ is of infinite type.  A {\it generalized Harish-Chandra module} is a finitely generated $\frak g$-module $M$ such that $M$ is a $(\frak g,\frak k)$-module of finite type for some Lie subalgebra $\frak k\subset\frak g$.

Note that given any integrable $\frak k$-module $E$, the induced $\frak g$-module $\text{ind}^{\frak g}_{\frak k} E$ is a strict $(\frak g, \frak k)$-module, however in general (and more specifically, for $\kk=\kk_0+C_\gg(\kk_0)$ as in \refsec{sec3} below) $\text{ind}^{\frak g}_{\frak k} E$ has infinite type\footnote{An interesting case when $\ind_\kk^\gg E$ has finite $\kk$-type is as follows. Using results of \cite{NP} it is easy to construct an embedding $\gl(\infty)\simeq\kk\subset\gg\simeq\gl(\infty)$, so that $\gg/\kk$ is isomorphic as a $\kk$-module to natural $\kk$-module $V$ (i.e. to the union of natural $\kk_n$-modules $V_n$, where $\kk_n\simeq\gl(n)$). Then  $\ind_\kk^\gg\CC\simeq S^\cdot(\gg/\kk)\simeq S^\cdot(V)$, and it is easy to see that the symmetric algebra is a multiplicity free $\kk$-module, i.e., in particular, $\ind_\kk^\gg$ has finite type as a $(\gg,\kk)$-module.}.  Therefore for the construction of strict simple $(\frak g,\frak k)$-modules of finite type, one needs more sophisticated techniques than ordinary induction. As we show below, cohomological induction is an ideal tool for this purpose.

Here are two examples illustrating the notions of a $(\frak g,\frak k)$-module of finite and of infinite type in the extreme case of an integrable $\frak g$-module.

\begin{prop}\label{prop13}  Let $\frak s = \cup_{n\in\ZZ_{>0}}\frak
s_n$ be any infinite dimensional locally simple Lie algebra and $\frak k_0 \subset\frak s_1$ be a finite dimensional subalgebra of $\frak s_1$.  Let $M$ be any non-trivial integrable $\frak s$-module.  Then $M$ is an $(\frak s,\frak k_0)$-module of infinite type.
\end{prop}

\begin{proof} Note first that $\dim M = \infty$.  This follows from
the fact that all $\frak s_n$ have no non-trivial common finite dimensional module since $\dim \frak s_n$ tends to $\infty$ when $n\to\infty$.  Now, assume to the contrary that $M$ is an $(\frak s, \frak k_0)$-module of finite type.  Then $M$ is a $(\frak s, \frak s_n)$-module of finite type for any $\frak s_n$.  We claim that this contradicts a result of Willenbring and Zuckerman. Indeed, Theorem 4.0.11 in \cite{WZ} implies that if the difference of dimensions $\dim \frak s_n - \dim \frak s_1$ is sufficiently large, then there is a finite number of simple finite dimensional $\frak s_1$-modules $W_1,\ldots, W_x$ such that any simple finite dimensional $\frak s_n$ module contains some $W_j$ as a $\frak s_1$-submodule.  It is an immediate consequence of this fact that any infinite dimensional $(\frak s, \frak s_n)$-module of finite type is an $(\frak s,\frak s_1)$-module of infinite type as some $W_j$ will appear with infinite multiplicity.  This contradiction shows that our assumption was false, i.e. $M$ is an $(\frak s,\frak k_0)$-module of infinite type.
\end{proof}

Let now $\frak g = \gl (p\Theta)$ where $\Theta = \theta_1\theta_2\ldots$ with $\theta_n > 1$ for all $n\in\ZZ_{>0}$, and let $\frak k_0 := \frak g_1 = \gl (p)$.  Set $\frak k_n := \frak k_0 + C_{\frak g_{n}} (\frak k_0)$ for $\frak g_n = \gl (p\theta_1\ldots \theta_{n-1})$, and $\frak k:= \cup_{n\in\ZZ_{>0}}\frak k_n$.  Then, as it is easy to check, $C_{\frak g_{n}} (\frak k_0) = \gl (\theta_1\ldots\theta_{n-1})$, and the inclusion $C_{\frak g_{n}} (\frak k_0)\subset C_{\frak g_{n+1}} (\frak k_0)$ is nothing but the $\theta_{n}$-diagonal inclusion.  Hence $\frak k\simeq \gl (p)+ \gl (\Theta)$.

\begin{prop}  The adjoint representation of $\gl (p\Theta)$ is a $C_{\frak g} (\frak k_0)$-module of finite length and thus, in particular, a $(\gl (p\Theta),\frak k)$-module of finite type.
\end{prop}

\begin{proof}  The statement follows from the observation that for each $n$, the adjoint representation of $\gl (p \theta_1\ldots \theta_{n-1})$ considered as a $C_{\frak g_{n}}(\frak k_0) = \gl (\theta_1\ldots \theta_{n-1})$-module is a submodule of $T^2(V^p_n \oplus (V^*_n)^p)$, where $V_n$ is the natural $\gl (\theta_1\ldots\theta_{n-1})$-module.  By \refprop{prop11}, the length of $T^2(V^p_n\oplus (V^*_n)^p)$ as an $\sl (\theta_1\ldots\theta_{n-1})$-module stabilizes for $n\to\infty$, hence the length of $\gl (p\theta_1\ldots\theta_{n-1})$ considered as a $C_{\frak g_{n}}(\frak k_0)$-module is bounded for $n\to\infty$. The reader will check immediately that this implies that the adjoint module of $\gl (p\Theta)$ has finite length as a $C_{\frak g}(\frak k_0)$-module.
\end{proof}

\subsection{The Zuckerman functor}  In this Subsection $\frak g$ is any Lie algebra and $\frak k'\subset \frak g$ is a finite dimensional subalgebra which acts locally finitely and semisimply on $\frak g$. For instance, if $\frak g = \cup_n\frak g_n$ is locally reductive and $\frak k'\subset\frak g_n$ is a reductive in $\frak g_n$ subalgebra for some $n$, the above condition is satisfied.

By $\mathcal{C}(\frak g,\frak k')$ we denote the category of all $(\frak g,\frak k')$-modules which are semisimple over $\frak k'$.  For any reductive in $\frak k'$ subalgebra $\frak m'\subset\frak k'$, we consider the left exact functor
\begin{eqnarray*}
\Gamma_{\frak{k}',\frak{m}'} : \mathcal{C} (\frak g, \frak{m}')&\to& \mathcal{C}(\frak g, \frak{k}') \\
M&\mapsto& \Gamma_{\frak{k}',\frak{m}'} (M) := \sum\limits_{X\subset M, X\in Ob(\mathcal{C}(\frak g,\frak{k}'))} X\,\,.
\end{eqnarray*}
The category $\mathcal{C}(\frak g, \frak m')$ has sufficiently many injectives and hence one can introduce the right derived functor $R^\cdot \Gamma_{\frak k',\frak m'}$.  This functor is known as \emph{the Zuckerman functor}.

A well known property of the Zuckerman functor which we use below is that if $Z_{U(\frak g)}$ acts via a fixed character on $M$, then $Z_{U(\frak g)}$ acts via the same character on $R^\cdot\Gamma_{\frak k', \frak m'} (M)$. The following two propositions discuss some further fundamental properties of the functor $R^\cdot\Gamma_{\frak k',\frak m'}$.

\begin{prop}\label{prop14}~
\item[(a)] (restriction principle).  Let $\frak g'\subset\frak g$ be an arbitrary Lie subalgebra of $\frak g$ such that $\frak k'\subset \frak g'$.  Then the diagram of functors
\begin{eqnarray*}
\xymatrix{
\mathcal{C} (\frak g, \frak m')\ar[d]\ar[r]^{R^\cdot\Gamma_{\frak k',\frak m'}}& \mathcal{C}(\frak g, \frak k') \ar[d]\\
\mathcal{C} (\frak g', \frak m')\ar[r]^{R^\cdot \Gamma_{\frak k',\frak m'} }& \mathcal{C}(\frak g',\frak k'), }
\end{eqnarray*}
whose vertical arrows are restriction functors, is commutative.
\item[(b)] Let $U^0(\frak k') :=\Gamma_{\frak k',\frak m'} (\Hom_{\CC} (U(\frak k'),\CC))$.  Then $U^0(\frak k')$ is a $U(\frak k')$-bimodule, and for any $M$ in $\mathcal{C} (\frak k',\frak m')$ there is a natural isomorphism of $\frak k'$-modules
$$
R^\cdot \Gamma_{\frak k',\frak m'} (M) \cong H^\cdot (\frak k',\frak m', M\otimes U^0(\frak k'))
$$
(here we apply $R^\cdot\Gamma_{\frak k'\frak m'}$ to objects of $\mathcal{C}(\frak k',\frak m')$ by setting $\frak g' = \frak k'$, see (a)).
\item[(c)] Let $M$ be an inductive limit $\varinjlim M_i$ of modules $M_i$ in $\mathcal{C} (\frak k', \frak m')$.  Then
$$
R^\cdot\Gamma_{\frak k',\frak m'} (M) \cong \varinjlim R^\cdot\Gamma_{\frak k',\frak m'} (M_i).
$$
\end{prop}

\begin{proof}  \item[(a)] It suffices to show that an injective object $I$ in
$\mathcal{C}(\frak g,\frak m')$ is also injective in $\mathcal{C}(\frak g', \frak m')$.  If $Q$ is an arbitrary object in $\mathcal{C}(\frak g', \frak m')$, then $U(\frak g)\otimes_{U(\frak g')} Q$ is an object in $\mathcal{C}(\frak g, \frak m')$, and the functor
$$
Q\mapsto U (\frak g)\otimes_{U (\frak g')} Q
$$
is exact.  The natural isomorphism $\text{Hom}_{\frak g} (U (\frak g)\otimes_{U\overline{} (\frak g')} Q, I)= \text{Hom}_{\frak g'} (Q,I)$ shows that $I$ represents an exact functor in $\mathcal{C}(\frak g',\frak m')$. Therefore $I$ is injective in $\mathcal{C}(\frak g', \frak m')$, and (a) follows.

\item[(b)] This statement is a rephrasing of the isomorphism (4.5) in \cite{EW}.

\item[(c)] For any $M$ in $\mathcal{C} (\frak k',\frak m')$, we use the standard complex for relative Lie algebra cohomology:
$$
C^\cdot (\frak k', \frak m', M\otimes U^0 (\frak k')) = \Hom_{\frak m'} (\Lambda^\cdot (\frak k'/\frak m'), M\otimes U^0 (\frak k')).
$$
As $\frak k'$ is finite-dimensional, we have an isomorphism
$$
C^\cdot(\frak k',\frak m', M\otimes U^0 (\frak k'))\simeq\varinjlim C^\cdot(\frak k',\frak m', M_i\otimes U^0 (\frak k')),
$$
and the fact that cohomology commutes with inductive limits implies (c).
\end{proof}

\begin{prop}\label{prop15} (comparison principle).  Suppose $\frak k' = \frak k''\oplus \frak k'''$ is a decomposition into two ideals, and let $\mm''$ be a reductive in $\frak k''$ subalgebra.  Set  $\frak m' := \frak m''\oplus\frak k'''$.  Then for any $(\frak g,\frak m')$-module $M$, there is a natural isomorphism of $\frak g$-modules
\begin{equation}\label{eq2}
R^\cdot\Gamma_{\frak k',\frak m'} (M) \simeq R^\cdot\Gamma_{\frak k'',\frak m''}(M).
\end{equation}
\end{prop}

\begin{lemma}\label{lemma16} Under the assumptions of \refprop{prop15}, let $I$ be an injective object in $\mathcal{C}(\frak g,\frak m')$.  Then
$$
R^t \Gamma_{\frak k'',\frak m''} (I) = 0 \,\, \text{for} \,\, t > 0.
$$
\end{lemma}

\textbf{Proof of \refle{lemma16}.} As a $\frak k'$-module $I$ can be decomposed as $\oplus_\lambda (J_\lambda \boxtimes V_{\frak k'''}(\lambda))$, where $\lambda$ runs over all dominant integral weights of $\frak k'''$ and where the $J_\lambda$'s are $(\frak k'',\frak m'')$-modules.  We claim that each $J_\lambda$ is injective in $\mathcal{C}(\frak k'',\frak m'')$.  Indeed, by the proof of the restriction principle (\refprop{prop14}(a)) $I$ is injective in $\mathcal{C} (\frak k',\frak m')$, hence for each $\lambda$, $J_\lambda\boxtimes V_{\frak k'''} (\lambda)$ is injective in $\mathcal{C}(\frak k', \frak m')$. Therefore $J_\lambda$ is injective in $C(\kk'',\mm'')$.

By \refprop{prop14}(b)
$$
R^\cdot \Gamma_{\frak k'',\frak m''} (I) \cong H^\cdot (\frak k'',\frak m'', I\otimes U^0(\frak k'')),
$$
and thus (since relative Lie algebra cohomology commutes with direct sums), it suffices to show that
\begin{equation}\label{eq03}
H^t(\frak k'',\frak m'', (J_\lambda \boxtimes V_{\frak k'''} (\lambda))\otimes U^0 (\frak k'')) = 0
\end{equation}
for $t > 0$. However,
\begin{equation*}
\begin{array}{l}
H^t (\frak k'',\frak m'', (J_\lambda\boxtimes V_{\frak k'''}(\lambda))\otimes U^0 (\frak k''))= \\
=H^t (\frak k'',\frak m'', J_\lambda\boxtimes U^0({\frak k''}))\boxtimes V_{\kk'''}(\lambda) =\\
=R^t\Gamma_{\kk'',\mm''}(J_\lambda)\boxtimes V_{\kk'''}(\lambda)=0
\end{array}
\end{equation*}
since $J_\lambda$ is injective in $C(\kk'',\mm'')$, and the Lemma follows. \qed

\textbf{Proof of \refprop{prop15}}  By \refle{lemma16}, any $\mathcal{C}(\frak g, \frak m')$-injective resolution of $M$ is $\Gamma_{\frak k'',\frak m''}$-acyclic hence it can be used both for the computation of $R^\cdot\Gamma_{\frak k',\frak m'}(M)$ and of $R^\cdot\Gamma_{\frak k'',\frak m''}(M)$.  This yields the natural isomorphism \refeq{eq2}. \qed
\medskip

\section{The Construction}\label{sec3}

Let $\frak g = \cup_n\frak g_n$ be a locally reductive Lie algebra and $\frak k_0\subset\frak g_1$ be a finite dimensional subalgebra reductive in $\frak g$ (equivalently, in $\frak g_1$). Fix a Cartan subalgebra $\frak t_0$ in $\frak k_0$. For any $\frak g_n$ we have the notion of a $\frak t_0$-{\it compatible parabolic subalgebra} of $\frak g_n$: by definition this is a parabolic subalgebra $\frak p_n\subset\frak g_n$ of the form $\bigoplus\limits_{\sigma,\text{Re} \sigma\geq 0} (\frak g_n)^\sigma_{h_{n}}$, where $h_n$ is a semisimple element of $\frak t_0,\sigma$ runs over the eigenvalues of $h_n$ in $\frak g_n$, and $(\frak g_n)^\sigma_{h_{n}}$ are the corresponding eigenspaces.  We call a subalgebra $\frak p\subset\frak g$ a {\it $\frak t_0$-compatible parabolic subalgebra} if, for all $n$, $\frak p\cap\frak g_n$ is a $\frak t_0$-compatible parabolic subalgebra of $\frak g_n$ and $\frak n_n = \frak n_{n+1}\cap\frak g_n$, where $\frak n_n$ is the nilradical of $\frak p_n$.  It is possible (but not required) that there is a semisimple element $h\in \frak t_0$ such that $\frak p = \bigoplus\limits_{\sigma, \text{Re} \sigma\geq 0} \frak g^\sigma_h$.

One can always choose decompositions $\frak p_n = \frak m_n \crplus \frak n_n$ where, for each $n$, $\frak m_n$ is a reductive in $\frak g_n$ subalgebra such that $\frak m_{n+1} \cap\frak g_n = \frak m_n$.  This yields a decomposition $\frak p = \frak m\crplus \frak n$, where $\frak m = \cup_n \frak m_n$ and $\cup_n\frak n_n$. By definition, $\frak n$ is the {\it nilradical} of $\frak p$ and $\frak m$ is a locally reductive subalgebra of $\frak g$.  In what follows, we consider the decomposition $\frak p = \frak m\crplus \frak n$ fixed and define $\bar{\frak n}$ as the union $\cup_n {\bar{{\frak n}}}_n$, where for each $n$, $\frak g_n = {\bar{{\frak n}}}_n\oplus \frak m_n\oplus{\frak n}_n$ is the canonical ${\frak m}_n$-module decomposition.  In this way, $\bar{\frak n}$ is of course an integrable $\frak m$-module.

%Any $\frak k_0$-compatible parabolic subalgebra $\frak p$ is of the form $\frak p = \frak m\oplus\frak n$, where $\frak m\cap\frak g_n$ and $\frak m\cap\frak g_n$ are respectively the reductive part and the nilradical of $\frak p_n$ for each $n\geq 1$.  The subalgebra $\frak m$ is a locally reductive subalgebra of $\frak g$, more precisely $\frak g_n\cap\frak m$ is a reductive in $\frak g_{n+1}\cap\frak m$ subalgebra of $\frak m$.  In what follows, we assume that $\frak p$ is a fixed $\frak k_0$-compatible parabolic subalgebra, and denote by $\frak{\bar m}$ the union $\cup_n \frak{\bar m}_n$, where for each $\frak g_n = \frak m_n\oplus\frak m_n\oplus\frak{\bar m}_n$ is the triangular decomposition determined by $\frak p_n$.

Let $\frak k := \frak k_0 + C_{\frak g} (\frak k_0)$.  Then $\frak k_n = \frak k_0 + C_{\frak g_{n}} (\frak k_0)$ is reductive in $\frak g$ for each $n$.  Note that $\frak k\cap\frak m = \frak m_0 + C_{\frak g} (\frak k_0)$, where $\frak m_0 := \frak k_0\cap\frak m$. Our goal is to construct nontrivial $(\frak g,\frak k)$-modules by starting with a nontrivial $(\frak m, \frak k\cap\frak m)$-module $E$ and then applying a functor of cohomological induction type. We first extend $E$ to a $\frak p$-module by setting $\frak n \cdot E = 0$. We then consider the induced module $M(\frak p, E) := \text{ind}^{\frak g}_{\frak p} E$.  This is an integrable $\frak m\cap\frak k$-module.  Indeed, the equality of $\frak m$-modules $\frak g = \bar{\frak n} \oplus \frak m\oplus\frak n$ implies via the Poincar$\acute{e}$-Birkhoff-Witt theorem that $M(\frak p, E)$ has an $\frak m$-module filtration with associated graded equal to $S^\cdot (\bar{\frak n})\otimes E$.  Both $S^\cdot(\bar{\frak n})$ and $E$ are integrable $\frak m\cap\frak k$-modules, thus $M(\frak p, E)$ is also $\frak m\cap\frak k$-integrable.

We now set $A(\frak p, E) := R^s \Gamma_{\frak k_{0},\frak m_{0}} (M(\frak p, E))$, where $s := \frac{1}{2} \dim (\frak k_0/\frak m_0)$.  By definition $A(\frak p, E)$ is a $(\frak g,\frak k_0)$-module, but as we show below $A(\frak p,E)$ is in fact a $(\frak g, \frak k)$-module.  We also set $A(\frak p_0, E):= R^s \Gamma_{\frak k_{0},\frak m_{0}} (\text{ind}^{\frak k_{0}}_{\frak p_{0}} E)$, where $\frak p_0 := \frak k_0\cap\frak p$ and we regard $E$ as a module over $\frak m_0 + C_{\frak g}(\frak k_0)$ and $\text{ind}^{\frak k_{0}}_{\frak p_{0}} E$ as a $\frak k_0 + C_{\frak g} (\frak k_0)$ module. By \refprop{prop14}(a) there is a functorial morphism of $\frak k_0$-modules
$$
\Psi_E: A(\frak p_0, E)\to A(\frak p, E).
$$
Knapp and Vogan \cite{KV} call $\Psi_E$ the {\it bottom layer map}. In the present paper, we call any $\frak g$-subquotient of $A(\frak p,E)$ generated by vectors in $\text{im} \Psi_E$ {\it a bottom layer subquotient} of $A(\frak p,E)$.

Note that $\frak m_0\cap C_{\frak g} (\frak k_0) = Z_{\frak k_{0}}$. Therefore, if $\bb_{\mm_0}$ is a fixed Borel subalgebra of $\mm_0$, we can decompose $E$ as
$$
\bigoplus\limits_\nu V_{{\frak m}_{0}} (\nu) \boxtimes_{U(Z_{\kk_0})} E''_\nu,
$$
where we consider $E''_\nu := \text{Hom}_{{\frak m}_{0}} (V_{{\frak m}_{0}} (\nu ), E)$ as a $C_{\frak g} (\frak k_0)$-module and $\nu$ runs over all $\bb_{\mm_0}$-dominant integral weights of $\frak m_0$.

Fix now a Borel subalgebra $\frak b_0$ of $\frak k_0$ such that $\frak b_0 \cap\mm_0=\bb_{\mm_0}$.  This defines two Weyl group elements:  the element $w_{\frak k_{0}}\in W_{\frak k_{0}}$ of maximal length with respect to $\frak b_0$, and the element $w_{\frak m_{0}}\in W_{\frak m_{0}}$ of maximal length with respect to $\frak b_0\cap \frak m_0$.  For any $\frak b_{\mm_0}$-dominant $\frak k_0$-integral weight $\nu$, we set
$$
\nu^{\vee} := w_{\frak k_{0}}\circ w^{-1}_{\frak m_{0}} (\nu + \rho_{\frak b_{0}}) - \rho_{\frak b_{0}},
$$
where $\rho_{\frak b_{0}}$ is the half-sum of the $\frak b_0$-positive roots of $\frak k_0$.

\begin{lemma} \label{le21} The $\frak k$-module $A(\frak p_0, E)$ is $\frak k$-integrable and is isomorphic to $\bigoplus\limits_\nu V_{\frak k_{0}} (\nu^{\vee})\boxtimes_{U(Z_{\kk_0})} E''_\nu$, where as above $\nu$ runs over all dominant integral weights of $\frak m_0$, and where $V_{\frak k_{0}}  (\nu^{\vee}) := 0$ whenever $\nu^{\vee}$ is not $\bb_0$-dominant and integral for $\frak k_0$.
\end{lemma}

\begin{proof} This statement is a direct corollary of the Bott-Borel-Weil theorem proved in \cite{EW}, see \cite[Proposition 6.3]{EW}. \end{proof}

The following theorem is our main result.

\begin{theo}\label{th22}

\item[(a)]  $A(\frak p, E)$ is a $(\frak g,\frak k)$-module.

\item[(b)]  If $M(\frak p, E)$ is an $(\frak m, \frak k\cap\frak m)$-module of finite type, then $A(\frak p, E)$ is a $(\frak g, \frak k)$-module of finite type.

\item[(c)] Assume $E = \cup_nE_n$ where each $E_n$ is an $(\frak m_n, \kk\cap\frak m_n)$-module on which $Z_{\frak m_{\frak n}}$ acts via a $1$-dimensional representation.  Then the bottom layer map $\Psi_E$ is an injection.  Assume that for some $\nu$, $E''_\nu\neq 0$ and $\nu^{\vee}$ is dominant integral for $\k_0$. Then $\Hom_{\frak k_{0}} (V_{\frak k_{0}}(\nu^{\vee}), A(\frak p, E)) =  E''_\nu$. Hence $A(\frak p, E)$ has a simple bottom layer subquotient.

\item[(d)]  Assume $E = \cup_n E_n$ where each $E_n$ is an $(\frak m_n, \kk\cap\frak m_n)$-module with $Z_{U(\frak m_{n})}$-character, that $A(\frak p,E)\neq 0$, and that for some $N$ the $Z_{U(\frak g_{N})}$ -character of $\ind^{\frak g_{N}}_{\frak p_{N}} E_N$ is not regular integral.  Then some bottom layer subquotient of $A(\frak p, E)$ is not an integrable $\frak g$-module.  If in addition, $\frak k$ is a maximal subalgebra of $\frak g$, then some simple bottom layer subquotient of $A(\frak p, E)$ is a strict $(\frak g, \frak k)$-module.

\item[(e)] Under the assumptions of (c) assume further that $\frak m = C_{\frak g} (\frak t_0)$ and that $E$ is simple. Then $\frak t_0$ acts via weight $\mu\in\frak t^*_0$ on $E, \, \mu^{\vee}$ is dominant integral for $\frak k_0$, and there is an isomorphism of $\frak k =\frak k_0 + C_{\frak g} (\frak k_0)$-modules
$$
A(\frak p_0, E)\simeq V_{\frak k_{0}} (\mu^{\vee})\boxtimes_{U(Z_{\frak k_{0}})} E''.
$$
where $E''$ equals $E$ considered as a $C_{\frak g} (\frak k_0)$-module.  Furthermore, $\Psi_E$ yields an isomorphism between the $\frak k$-modules $A(\frak p_0, E)$ and $V_{\frak k_{0}} (\mu^{\vee})\otimes \Hom_{\frak k_{0}} (V_{\frak k_{0}} (\mu^{\vee}), A(\frak p,E))$.

\item[(f)] If, under the assumptions of e), $\text{im} \Psi_E$ is a simple $\frak k$-submodule of $A(\frak p,E)$, then $A(\frak p,E)$ has a unique simple bottom layer subquotient.  A sufficient condition for the simplicity of $\text{im} \Psi_E$ is the inclusion $\frak m\subset\frak k$.
\end{theo}

\begin{proof}
\item[(a)] By construction, $M(\frak p, E)$ is a $(\frak g,\kk\cap\mm)$-module. Since $\frak m\cap\frak k\supset C_{\frak g} (\frak k_0), M(\frak p,E)$ is an integrable $C_{\frak g} (\frak k_0)$-module. Let $\tilde M$ denote the restriction of $M(\frak p,E)$ to $\frak k$:  by \refprop{prop14}(a)  $A(\frak p, E)$ is isomorphic as a $\frak k$-module to $R^s \Gamma_{\frak k_{0},\frak m_{0}} (\tilde M)$. By the Poincar\'e-Birkhoff-Witt Theorem, the $\frak k$-module $\tilde M$ has an increasing filtration with associated graded
\begin{equation}\label{eq4}
\mathrm{Gr} \,\, \tilde M = \bigoplus\limits_{t\in\ZZ_{\geq 0}}\, \text{ind}^{\frak k_{0}}_{\frak p_{0}} (S^t (\frak k^c_0\cap\bar{\frak n})\otimes E),
\end{equation}
where $\frak k^c_0$ is a fixed $\frak k_0$-module complement of $\frak k_0$ in $\frak g$.
\end{proof}

\begin{lemma}\label{le23}  $R^\cdot \, \Gamma_{\frak k_{0},\frak m_{0}} (\Gr
\, \tilde M)$ is a graded integrable $\frak k$-module. \end{lemma}

\textbf{Proof of \refle{le23}.} Decompose the ${\frak m}_0+ C_{\frak g} (\frak k_0)$-module $S^t (\frak k_0^c\cap\bar{\frak n})\otimes E$ as
$$
\bigoplus\limits_\nu \, V_{\frak m_{0}} (\nu) \boxtimes_{U(Z_{\kk_0})} \, X_{\nu,t}
$$
for some $C_{\frak g} (\frak k_0)$-modules $X_{\nu,t}$.  Observe that each $X_{\nu,t}$ is an integrable $C_{\frak g} (\frak k_0)$-module.  We obtain a $\frak k$-module isomorphism
$$
\Gr \,\, \tilde M \cong \bigoplus\limits_{\nu, t} \text{ind}^{\frak k_{0}}_{\frak p_{0}}\left( V_{\frak m_{0}} (\nu) \, \boxtimes_{U(Z_{\kk_0})} \, X_{\nu,t}\right).
$$
For each $\nu$, let $G^{\cdot}_\nu$ be a resolution of $\text{ind}^{\frak k_{0}}_{\frak p_{0}} \, V_{\frak m_{0}} (\nu)$ by $\Gamma_{\frak k_{0},\frak m_{0}}$-acyclic $(\frak k_0,\frak m_0)$-modules.  We can compute $R^\cdot \, \Gamma_{\frak k_{0},\frak m_{0}} (\Gr \, \tilde M)$ as
$$
H^\cdot (\Gamma_{\frak k_{0},\frak m_{0}} (\bigoplus\limits_{\nu,t} G^{\cdot} _\nu \boxtimes_{U(Z_{\kk_0})} \, X_{\nu,t})),
$$
which is isomorphic as a $\frak k$-module to
$$
\bigoplus\limits_{\nu,t} H^\cdot( \Gamma_{\frak k_{0},\frak m_{0}} (G^{\cdot}_{\nu,t}))\, \boxtimes_{U(Z_{\kk_0})} \, X_{\nu,t}
$$
and hence to
\begin{equation}\label{eq5}
\bigoplus\limits_{\nu, t} R^\cdot \Gamma_{\frak k_{0},\frak m_{0}} (V_{\frak m_{0}} (\nu ))\boxtimes_{U(Z_{\kk_0})} X_{\nu,t}.
\end{equation}
Therefore $R^\cdot \, \Gamma_{\frak k_{0},\frak m_{0}}(\Gr \tilde M)$ is an integrable $\frak k$-module.  This proves the Lemma. \qed

To complete the proof of (a) note that, by \refprop{prop14}(c), $R^\cdot\Gamma_{\frak k_{0},\frak m_{0}}$ commutes with inductive limits.  Since furthermore, $\mathcal{C}_{\frak g}(\frak k_0)$ acts by $\frak k_0$-endomorphisms on $\tilde M$, $R^\cdot\Gamma_{\frak k_{0},\frak m_{0}} (\tilde M)$ has an increasing filtration of $\frak k_0 + \mathcal{C}_{\frak g}(\frak k_0)$-modules induced by the filtration on $\tilde M$. An obvious induction argument using the fact that $R^\cdot\Gamma_{\frak k_{0},\frak m_{0}} (\Gr\tilde M)$ is a $\frak k$-integrable module (\refle{le23}) implies that $R^\cdot\Gamma_{\frak k_{0},\frak m_{0}} (\tilde M)$ is filtered by $\frak k$-integrable modules, and hence is itself $\frak k$-integrable.  This proves (a).

\item[(b)]  Suppose $M(\frak p, E)$ is of finite type over $\frak k\cap\frak m =  {\frak m}_0+ C_{\frak g} ({\frak k}_0)$. We can rewrite \refeq{eq4} as
$$
\Gr \, \tilde M = \bigoplus\limits_\nu  (\text{ind}^{\frak k_{0}}_{\frak p_{0}} \, V_{\frak m_{0}} (\nu)) \, \boxtimes_{U(Z_{\kk_0})} \, Y_\nu
$$
with each $Y_\nu = \oplus_t X_{\nu, t}$ an integrable $C_{\frak g}(\frak k_0)$-module.  Since $\text{ind}^{\frak k_{0}}_{\frak p_{0}} \, V_{\frak m_{0}} (\nu)$ is a $(\frak k_0,\frak m_0)$-module, we conclude that every $Y_\nu$ is of finite type over $C_{\frak g}(\frak k_0)$.  Combining \refeq{eq5} with \refle{le23}, we obtain
\begin{equation}\label{eq6}
R^s\Gamma_{\frak k_{0},\frak m_{0}} (\Gr \tilde M) \cong \bigoplus\limits_\nu V_{\frak k_{0}} (\nu^{\vee})\boxtimes_{U(Z_{\kk_0})} Y_\nu .
\end{equation}
The right hand side of \refeq{eq6} is of finite type over $\frak k$ as each $Y_\nu$ is of finite type over $C_{\frak g} (\frak k_0)$ and $V_{\frak k_{0}} (\nu'\,^{\vee}) \not\cong V_{\frak k_{0}} (\nu'' \,^{\vee})$ for $\nu'\neq \nu''$.  Finally, the fact that $R^s\Gamma_{\frak k_{0},\frak m_{0}} (\Gr \tilde M)$ is of finite type over $\frak k$ implies that $R^s \Gamma_{\frak k_{0},\frak m_{0}} (M)$ is of finite type over $\frak k$.  Indeed, this follows from the observation, that since $R^s \Gamma_{\frak k_{0},\frak m_{0}}$ commutes with inductive limits,
\begin{equation}\label{eq7}
\Gr (R^s\Gamma_{\frak k_{0},\frak m_{0}} (\tilde M)) \cong R^s \Gamma_{\frak k_{0},\frak m_{0}} (\Gr \tilde M),
\end{equation}
where the left hand side of \refeq{eq7} refers to the filtration of $R^s \Gamma_{\frak k_{0},\frak m_{0}}(\tilde M)$ induced by the filtration on $\tilde M$. This proves (b).

\item[(c)]  The theory of the bottom layer map in the finite dimensional case is elaborated by Knapp and Vogan in \cite[Ch.\romanV, Sec.6]{KV}.  There the authors assume that they are working with a symmetric pair.  However, a careful examination of Theorem 5.80 in \cite{KV} reveals that the assumption that $\frak k_0$ is symmetric in $\frak g_n$ is not needed; hence our hypothesis on $E_n$ implies that $\Psi_{E_{n}}$ is an injection from $A(\frak p_0, E_n)$ to $A(\frak p_n, E_n)= R^s \Gamma_{\frak k_{0},\frak m_{0}} (\text{ind}^{\frak g_{n}}_{\frak p_{n}}\, E_n)$ for each $n$.  Furthermore, we have an injection of $\text{ind}^{\frak g_{n}}_{\frak p_{n}}\, E_n$ to $\text{ind}^{\frak g_{n+1}}_{\frak p_{n+1}}\, E_{n+1}$ which induces a $\frak g_n$-module homomorphism $\phi_n : A(\frak p_n, E_n)\to A(\frak p_{n+1}, E_{n+1})$.

On the other hand, we have a canonical $\kk_0$-module homomorphism $\chi_n :A(\frak p_0, E_n)\to A(\frak p_0, E_{n+1})$  induced by the inclusion of $E_n$ into $E_{n+1}$. Moreover, the diagram
\begin{equation}\label{eq8}
\xymatrix{
A(\frak p_0, E_{n+1}) \ar[r]^{\Psi_{E_{n+1}}}& A(\frak p_{n+1}, E_{n+1})  \\
A(\frak p_0, E_n) \ar[u]_{\chi_n}\ar[r]^{\Psi_{E_{n}}}&A(\frak p_n, E_n)\ar[u]_{\phi_n}}
\end{equation}
is commutative, and $\Psi_{E_{n}}$ and $\Psi_{E_{n+1}}$  are injections.  Consider the inductive limit homomorphism
$$
\varinjlim \Psi_{E_{n}} : \varinjlim A(\frak p_0, E_n)\to \varinjlim A(\frak p_n, E_n).
$$
By \refprop{prop14}(c) $\Psi_E = \varinjlim \Psi_{E_{n}}$ is an injection.

Assume now that for some $\nu$, $E''_\nu\neq 0$ and $\nu^{\vee}$ is dominant integral for $\k_0$.  For sufficiently large $n$, $E''_{n,\nu}:= \Hom_{{\frak m}_{0}}(V_{{\frak m}_{0}} (\nu), E_n)$ is always nonzero.  The fact that $\Hom_{\frak k_{0}} (V_{\frak k_{0}} (\nu^{\vee}), A(\frak p_n,E_n)) \cong \Hom_{\frak k_{0}} (V_{ \frak k_{0}}(\nu^{\vee}), A(\frak p_0, E_n))$ (\cite[Theorem 5.80]{KV}), together with the fact that $\Psi_E = \varinjlim \Psi_{E_{n}}$, implies
$$
\Hom_{\frak k_{0}}(V_{\frak k_{0}} (\nu^{\vee}), A(\frak p,E)) = E''_\nu
$$
as required. In particular, the bottom layer $\Im \Psi_E\subset A(\pp,E)$ is non-zero. Finally, to construct a simple bottom layer quotient of $A(\pp,E)$ it suffices to consider a simple quotient of a cyclic module $U(\gg)\cdot v$, where $v\in\Im\Psi_E$. This proves (c).

%By our assumption, for some $n$, the $Z_{U(\frak g_{n})}$-character of
%$A(\frak p_n, E_n)$ is not regular integral.  Fix a non-zero vector
%$v\in A(\frak p_0, E)$.  Let $\tilde v = \Psi_E (v)$, let $A_v$ be
%the $\frak g$-submodule of $A(\frak p,E)$ generated by $\tilde v$,
%and $A'_v$ be a  simple quotient of $A_v$.  We claim that $A'_v$ is
%not $\frak g$-integrable.  To see this, let $A'_{v,n}$ be the image
%of the $\frak g_n$-submodule of $A(\frak p, E)$ generated by $\tilde
%v$.  Since $A'_{v,n}$ is isomorphic to a quotient of $A(\frak p_n,
%E_n), A'_{v,n}$ is not $\frak g_n$-integrable.  Hence $A'_v$ is not
%$\frak g_n$-integrable and thus also not $\frak g$-integrable.

For the proof of (d) we need the following lemma.

\begin{lemma}\label{le24}  Suppose $F$ is an integrable $\frak m_0$-module.  Extend $F$ to a $\frak p_0$-module so that $\frak n_0\cdot
F = 0$.  Then if $i < s$, $R^i \Gamma_{\frak k_{0},\frak m_{0}}(\ind^{\frak k_{0}}_{\frak p_{0}} F) = 0$.
\end{lemma}

\textbf{Proof of \refle{le24}.} According to \refprop{prop14}(b) we need to show that
$$
H^i(\frak k_0,\frak m_0, (\text{ind}^{\frak k_{0}}_{\frak p_{0}} F)\otimes U^0(\frak k_0)) = 0
$$
for $i < s$.  Since $U^0(\kk_0)$ is a semisimple integrable $\kk_0$-module, it is enough to show that  $H^i (\frak k_0, \frak m_0 , V\otimes \text{ind}^{\frak k_{0}}_{\frak p_{0}} F) = 0$ for $i < s$ and for any simple finite-dimensional $\kk_0$-module $V$.  By Poincar\'e duality for relative Lie algebra cohomology we must show that
$$
H_{2s-i} (\frak k_0,\frak m_0, V \otimes \text{ind}^{\frak k_{0}} _{\frak p_{0}} F) = 0
$$
for $i < s$. It is well known that
$$
V\otimes \text{ind}^{\frak k_{0}} _{\frak p_{0}} F \cong \text{ind}^{\frak k_{0}} _{\frak p_{0}}(V\otimes F).
$$
So we must show that
$$
H_{2s-i} (\frak k_0, \frak m_0, \text{ind}^{\frak k_{0}} _{\frak p_{0}}(V\otimes F)) = 0
$$
for $i < s$.  But Shapiro's Lemma implies that the above homology is isomorphic to $H_{2s-i} (\frak p_0, \frak m_0, V\otimes F)$, and the latter vanishes for  $i < s$ because $\dim (\frak p_0 / \frak m_0) = s$.  The Lemma follows.  \qed

\item[(d)] Consider the short exact sequence
$$
0\to\text{ind}^{\frak k_{0}} _{\frak p_{0}}  E_n \to\text{ind}^{\frak k_{0}} _{\frak p_{0}} E_{n+1} \to\text{ind}^{\frak k_{0}} _{\frak p_{0}} (E_{n+1}/E_n) \to 0.
$$
It yields a long exact sequence for $R^\cdot \Gamma_{\frak k_{0},\frak m_{0}}$.  \refle{le24} implies that each $\chi_n$ is an injection.  Therefore, by the commutativity of diagram \refeq{eq8}, $\phi_n\circ \Psi_{E_{n}}$ is an injection for each $n$, and hence the maps $\phi_n\circ\Psi_{E_{n}}$ induce an injection
$$
i_n : A(\frak p_0, E_n) \to A(\frak p, E)
$$
for each $n$.

Fix a value of $N$ so that $A(\frak p_0, E_N)\neq 0$, and so that the $Z_{U(\frak g_{N})}$-character of $\text{ind}^{\frak g_{N}} _{\frak p_{N}} E_N$ is not regular integral.  Fix a nonzero vector $v\in A(\frak p_0, E_N)$, let $A_v$ be the $\frak g$-submodule generated by $\tilde v  := \Psi_E (i_n(v))$ (note that $\tilde v\neq 0$), and let $A'_v$ be a simple quotient of $A_v$.  We claim that $A'_v$ is not $\frak g$-integrable.  To see this consider the image $A'_{v,N}$ in $A'_v$ of the $\frak g_N$-submodule $U(\frak g_N)\cdot\tilde v \subset A(\frak p,E)$.  The commutativity of the diagram
$$\xymatrix{
A(\frak p_0, E) \ar[r]^{\Psi_E}& A(\frak p,E) \\
A(\frak p_0, E_N) \ar[u]^{i_N}\ar[r]^{\Psi_{E_{N}}} &A(\frak p_N, E_N)\ar[u] }
$$
implies that $A'_{v,N}$ is isomorphic to a subquotient of $A(\frak p_N,E_n)$.  Since  $Z_{U(\frak g_{N})}$ acts by one and the same character on $\text{ind}^{\frak g_{N}}_{\frak p_{N}} E_N$ and on $A(\frak p_N, E_N), A'_{v,N}$ is a $\frak g_N$-module with a central character which is not regular integral, and is thus not an integrable $\frak g_N$-module.  This implies that $A'_v$ itself is not an integrable $\frak g$-module.

\item[(e)] Note that, under our assumptions, $\frak m_0 = \frak t_0$.  As $\frak t_0\subset Z_\m$, $\t_0$ acts via weight $\mu$ on $E$, and moreover, $E = \CC_\mu \boxtimes_{U(Z_{\frak k_{0}})} E''$ where $\CC_\mu$ is the $1$-dimensional $\frak t_0$-module corresponding to $\mu$.  \refle{le21} yields now (3), and (c) implies that $\Psi_E$ is an isomorphism between $A(\frak p_0,E)$ and $V_{\frak k_{0}} (\mu^{\vee})\otimes\Hom_{\frak k_{0}} (V_{\frak k_{0}} (\mu^{\vee}), A(\frak p,E))$.

\item[(f)] Assume in addition that $\text{im} \Psi_E$ is a simple $\frak k$-module.  Let $A^{\#}$ denote the $\frak g$-submodule of $A(\frak p,E)$ generated by $\text{im} \Psi_E$, and let $A^{\$}$ be the sum of all $\frak g$-submodules $X$ of $A^{\#}$ with $\Hom_{\frak k_{0}} (V_{\frak k_{0}} (\mu^{\vee}), X) = 0$.  Then (e) together with the $\frak k_0$-semisimplicity of $A(\frak p,E)$ imply that $A^{\$} $ is a maximal proper $\frak g$-submodule of $A^{\#}$, and hence $A^{\#}/A^{\$} $ is the unique bottom layer subquotient of $A(\frak p,E)$.

Finally, the inclusion $\frak m\subset \frak k$ yields $\frak m = C_{\frak g} (\frak t_0)\subset \frak k_0+ C_{\frak g} (\frak k_0)$ which implies that  $\frak m = \frak t_0 + C_{\frak g} (\frak k_0)$.  As $\frak t_0$ is abelian, $E''$ is a simple $C_{\frak g} (\frak k_0)$-module, and the isomorphism \refeq{eq03} of (e) implies that $A(\frak p_0, E)$ is a simple $\frak k$-module.  Therefore (by (c)) $\text{im} \Psi_E$ is isomorphic to $A(\frak p_0, E)$, and is thus a simple $\frak k$-module.  \qed

In the spirit of \cite{PSZ} we call a locally reductive subalgebra $\frak{l}\subset\frak g$ of a locally reductive Lie algebra $\frak g$ {\it primal}, if there exists a simple strict $(\frak g,\frak l)$-module $M$ such that $\frak l$ is a maximal locally reductive subalgebra of $\frak g [M]$.  Using \refth{th22}, one can prove that certain subalgebras $\frak{l}$ are primal, for instance a subalgebra $\frak k = \frak k_0 + C_{\frak g} (\frak k_0)$ is primal whenever there exists an $\frak m$-module $E$ satisfying the assumption of \refth{th22}(d). Below we show the primality of $\frak k$ in some special cases.

\section{The case $\frak g = \gl (p\Theta)$}
To illustrate our main result in the specific case of $\frak g = \gl (p\Theta)$, fix the exhaustion $\frak g = \cup_n \gl (p\theta_1\ldots\theta_{n-1})$ as in Subsection \ref{sec23}.  Let $\frak k_0\subset\frak g_1 = \gl (p)$ be any reductive in $\frak g_1$ subalgebra which contains a $\frak g_1$-regular element $h$, and such that the $p$-dimensional natural $\gl (p)$-module $\CC^p$ is simple as a $\frak k_0$-module.  For instance, $\frak k_0$ may equal $\gl (p)$, $s\ell (p)$ or a principal $s\ell (2)$-subalgebra of $s\ell (p)$. Let $\frak t_0 := {C}_{\frak k_{0}} (h)$. We define $\frak p$ as the $\frak t_0$-compatible parabolic subalgebra $\bigoplus\limits_{\sigma ,\text{Re} \sigma\geq 0} \frak g^\sigma_h$.

\begin{lemma}\label{le31}
\item[(a)] $\frak m\cap\frak g_n\simeq \gl(\theta_1\ldots\theta_{n-1})^p$.

\item[(b)] $C_{\frak g_{n}} (\frak k_0)\simeq \gl (\theta_1\ldots\theta_{n-1})$ is the diagonal subalgebra in $\gl (\theta_1\ldots\theta_{n-1})^p$.
\end{lemma}

\begin{proof} As an $C_{\gg_n}(h)$-module, the natural representation
$V_n$ of $\gl (p\theta_1\ldots\theta_{n-1})$ decomposes as a direct sum of $p$ isotypic components each of dimension $\theta_1\ldots\theta_{n-1}$.  This yields (a).

As a $\frak k_0$-module $V_n$ decomposes as a direct sum of $\theta_1\ldots\theta_{n-1}$ copies of the simple $\frak k_0$-module $\CC^p$.  This implies (b).
\end{proof}

\begin{corollary}\label{cor32}
\item[(a)] $\frak m = {C}_{\frak g}(\frak t_0) = \gl(\Theta)^p$;
\item[(b)] $\frak k\simeq\frak k_0 + \gl (\Theta)$, $\kk_0\cap\gl(\Theta)\subset Z_{\gl(\Theta)}$;
\item[(c)] if $\frak k_0 = \gl (p)$, then $\frak k \simeq \gl (p)+\gl(\Theta)$ is a maximal proper subalgebra of $\gl (p\Theta )$.
\end{corollary}

We now construct a class of simple $\gl (\Theta)$-modules. Let $V_n$ denote the natural representation of $\gl(\theta_1\dots\theta_{n-1})$. Fix $n_0>1$ and let $V(\lambda_{n_{0}})$ be the simple finite dimensional $\gl (\theta_1\ldots\theta_{n_{0}-1})$-module with highest weight $\lambda_{n_{0}} = (\lambda^1,\ldots,\lambda^{\theta_{1}\ldots\theta_{n_{0}-1}}), \lambda^i\geq\lambda^{i+1}$.  Define $n' = n'(\lambda_{n_{0}-1})$ as the largest index for which the entry $\lambda^{n'}$ is non-negative; if $\lambda^1 < 0$, we put $n' = 0$. To $\lambda_{n_{0}}$ we assign the following highest weight of $\gl (\theta_1\ldots\theta_{n_{0}})$:
$$
\lambda_{n_{0}+1} := (\lambda^1,\ldots ,\lambda^{n'},\underbrace{0,0,\ldots,0,}_{\theta_1\ldots\theta_{n_{0}}(\theta_{n_{0}+1}-1)\text{times}} \lambda^{n'+1},\ldots,\lambda^{\theta_{1}\ldots\theta_{n_{0}}}).
$$

\begin{lemma}\label{le33} There is a natural injection of $\gl(\theta_1\ldots\theta_{n_{0}-1})^{\theta_{n_{0}}}$-modules
$$
V(\lambda_{n_{0}})^{\theta_{n_{0}}} \to  V(\lambda_{n_{0}+1}),
$$
and hence a diagonal injection of $\gl (\theta_1\ldots\theta_{n-1})$-modules
$$
V(\lambda_n)\to V(\lambda_{n+1})
$$
for any $n>n_0$.
\end{lemma}
\begin{proof}  The natural injection $V_{n_{0}}^{\theta_{n_{0}}}\to V_{n_{0}+1}$ induces a natural injection of $\gl (\theta_1\ldots\theta_{n_{0}})^{\theta_{n_{0}+1}}$-modules
$$
T^\cdot(V_{n_{0}}\oplus V^*_{n_{0}})^{\theta_{n_{0}+1}}\to T^\cdot(V_{n_{0}+1}\oplus V^*_{n_{0}+1})
$$
which in turn induces an injection
$$
V(\lambda_{n_{0}})^{\theta_{n_{0}}}\to V(\lambda_{n_{0}+1})
$$
as required.
\end{proof}

\begin{corollary}\label{cor34} For every $n_0$ and any dominant integral
weight $\lambda_{n_{0}}$ of $\gl (\theta_1\ldots\theta_{n_{0}-1})$, $\tilde V(\lambda_{n_{0}})$ is a simple $\gl (\Theta)$-module defined as the direct limit $\varinjlim\limits_{n\geq n_{0}} V(\lambda_n)$, where $V(\lambda_n)$ is embedded diagonally into $V(\lambda_{n+1})$ according to \refle{le33}.
\end{corollary}

Let now $\lambda_{n^{1}_{0}}\ldots \lambda_{n^{p}_{0}}$ be $p$ dominant weights as in \refcor{cor34}.  Assume that the ordering of the weights is compatible with $\frak n$, i.e. that the $h$ value of any root $\varepsilon_i - \varepsilon_j, i < j$, of $\frak g_1 = \gl (p)$ has non-negative real part.  Define $E$ as $V(\lambda_{\frak n^{1}_{0}})\boxtimes\ldots\boxtimes\tilde V(\lambda_{n^{p}_{0}})$ with trivial action of $\frak n$.

\begin{prop}\label{prop35} $M(\frak p, E)=\ind_\pp^\gg E$ is an $(\frak m,\frak k\cap\frak m)$-module of finite type.
\end{prop}

\begin{proof} It suffices to show that $\Gr \, M(\frak p, E)$ is an $(\frak m, \frak k\cap\frak m)$-module of finite type.  As a $\frak m$-module $\Gr \, M(\frak p, E)$ is isomorphic to $S^\cdot (\bar{\frak n})\otimes E$, and is in particular a weight module over the Cartan subalgebra $\frak t_0$ of $\frak k_0$.  This subalgebra acts via a single weight on $E$ and via arbitrary sums of $\frak p$-negative $\frak t_0$-weights on $S^\cdot (\bar{\frak n})$.  Since each $\frak t_0$-weight of $S^\cdot (\bar{\frak n})$ occurs only in finitely many symmetric powers of $\bar{\frak n}$, it suffices to show that each fixed tensor product $S^t(\bar{\frak n})\otimes E$ is a $\frak k\cap\frak m$-module of finite length.  Notice that $E$ is a direct limit $\varinjlim\limits_{\frak n\geq\text{max}(n^{1}_{0},\ldots\frak n^{p}_{0})} E_n$ such that each $E_n$ is a  $C_{\frak g_{n}} (\frak k_0)\simeq \gl (\theta_1\ldots\theta_{n-1})$-submodule of a fixed tensor power $T^k(V^p_n\oplus (V^*_n)^p)$.  Hence $S^t(\bar{\frak n}_{n})\otimes E_n$ is also contained in a fixed tensor power $T^{k} (V^p_n\oplus (V^*_n)^p)$.  \refprop{prop11} now implies that, for each $n$, $S^t(\bar{\frak n}_{n})\otimes E_n$ is a $C_{\frak g} (\frak k_0)\cap\frak g_n$-module of finite length, hence $S^t(\bar{\frak n})\otimes E$ is a $\frak k\cap\frak m$-module of finite length. The Proposition follows.
\end{proof}

Note now that the assumptions of \refth{th22}(e) apply to the case we consider.  Therefore, to ensure that $A(\frak p, E)$ is non-zero, it suffices to ensure that the weight $\mu^{\vee}$ is integral $\frak k_0$-dominant.  An easy computation shows that the weight $\mu$ is nothing but the weight $(\sum\limits_i \lambda^i_{n^{1}_{0}}, \sum\limits_i \lambda^i_{n^{2}_{0}}, \ldots , \sum\limits_i \lambda^i_{n^{p}_{0}})$ of $\frak g_1$, restricted to $\frak t_0$. Let $\kk_0=\gl(p)$. Then the regularity and $\frak k_0$-dominancy condition on $\mu^{\vee}$ are equivalent to the condition
$$
\sum\limits_i \lambda^i_{n^{1}_{0}}\leq \sum\limits_i \lambda^i_{n^{2}_{0}}\leq\ldots\leq \sum\limits_i \lambda^i_{n^{p}_{0}}.
$$

Note furthermore, that our choice of weights $\lambda_{n^{1}_{0}},\ldots ,\lambda_{n^{p}_{0}}$ allows for the possibility the $Z_{U(\frak g_{N})}$-character of $\text{ind}^{\frak g_{N}}_{\frak p_{N}}E_N$ to be non-regular for some $N$, and hence in the latter case, no irreducible bottom layer quotient of $A(\frak p, E)$ is $\frak g$-integrable. Since $\frak k_0 = \gl (p)$, $\frak k$ is a maximal proper subalgebra of $\gl (p\Theta)$. This implies (via \refth{th22}(d)) that whenever $A(\frak p, E)$ is not integrable, any irreducible bottom layer quotient of $A(\frak p ,E)$ is a strict $(\frak g, \frak k )$-module. In particular, $\frak k=\gl(p)+\gl(\Theta)$ is a primal subalgebra of $\gl (p\Theta)$.

Finally, \refle{le31} (a) and (b) imply that the condition $\frak m\subset\frak k$ from \refth{th22}(f) holds only when $p=1$.  However, in this case $s=0$, hence the claim of (f) is trivial.  Nevertheless, there is an interesting non-trivial case in which \refth{th22} (f) applies: this is when $\lambda_{\frak n^{1}_{0}} =\ldots = \lambda_{n^{p-1}_{0}} = 0$ and $\lambda_{n_{0}^{p}}\neq 0$.  In this latter case $E''$ is clearly a simple ${C}_{\frak g}(\frak k_0)$-module.  Furthermore, as it is easy to see, for large $n$ the $Z_{U(\frak g_{n})}$-character of $\text{ind}^{\frak g_{n}}_{\frak p_{n}} E_n$ is integral but not regular, hence the $(\gg,\kk)$-module $A(\frak p,E)$ has a unique strict simple subquotient.

\section{The root-reductive case}\label{TheRootRed}
Let now $\frak g$ be a simple infinite dimensional root-reductive Lie algebra, i.e. $\frak g\cong s\ell (\infty), so(\infty), sp(\infty)$.  Fix an exhaustion $\frak g = \cup_n \frak g_n$, where $\frak g_n\subset\frak g_{n+1}$ is a root injection of the form $sl(i)\subset s\ell (i+1)$, $so(i)\subset so(i+2)$, or $sp(2i)\subset sp(2i+2)$, for $\frak g$ isomorphic respectively to $s\ell(\infty), so(\infty)$ or $sp(\infty)$. Then each $\frak g_n$ is reductive in $\frak g$ and $C_{\frak g}(\frak g_n)\simeq \frak g$ for $\frak g \simeq so(\infty), sp(\infty)$, and $C_{\frak g}(\frak g_n)\simeq \gl (\infty)$ for $\frak g = s\ell (\infty)$.  Moreover, for a fixed $n$, the subalgebra $\frak g_n \oplus C_{\frak g}(\frak g_n)$ has the property that its intersections with $\frak g_{n'}$ for all $n' > n$ are symmetric subalgebras. 

We fix next a reductive in $\frak g_1$ subalgebra $\frak k_0\subset\gg_1$, a Cartan subalgebra $\frak t_0\subset \frak k_0$ and a $\frak t_0$-compatible parabolic subalgebra $\frak p = \frak m\crplus\frak n$, and let $\frak m_0 = \frak m\cap\frak k_0$.  For instance, for $\gg\simeq\sl(\infty)$, $\frak p$ can be a maximal proper subalgebra of $\frak g$, whose intersection with $\frak g_n$ for $n> 1$ equals a maximal parabolic subalgebra of $\frak g_n$ containing $C_{\frak g_{n}}(\frak g_1)$. Note that
\begin{equation}\label{eq9}
\frak m_0\oplus C_{\frak  g} (\frak g_1)\subset \frak k\cap \frak m.
\end{equation}

Let $E = \cup_n E_n$, where, for $n$ large enough, each $E_n$ is a simple $\frak m_n$-submodule of a tensor power $T^k (V^a_n\oplus (V^*_n)^b\oplus\CC^c)$ for fixed $k,a,b,c$ (when $\frak g\simeq so(\infty), sp(\infty)$, there is an isomorphism $V_n\simeq V^*_n$).

\begin{prop} \label{prop42} $M(\frak p, E)$ is an $(\frak m,\frak
k\cap\frak m)$-module of finite type.
\end{prop}

\begin{proof} According to \refeq{eq9}, it suffices to show that $M(\frak p,E)$ is an $\frak m_0\oplus C_{\frak g}(\frak g_1)$-module of finite type.  The argument is very similar to that in the proof of \refprop{prop35}. Consider $\Gr M(\frak p,E)\simeq S^\cdot (\bar{\frak n})\otimes E$ and note that only finitely many $\frak t_0$-weights occur in $E$, and that each $\frak t_0$-weight of $S^\cdot(\bar{\frak n})$ will occur only in finitely many symmetric powers of $\bar{\frak n}$.  Hence it suffices to show that each fixed tensor product $S^t(\bar{\frak n})\otimes E$ is a $C_{\frak g}(\frak g_1)$-module of finite length.  However, a direct verification based on the definition of $\frak g_1$ shows that for each $n > 1, \bar{\frak n}\cap\frak g_n$ is a $C_{\frak g}(\frak g_1)\cap\frak g_n$-submodule of a fixed tensor power $T^{k}(V^{a}_n\oplus (V^*_n)^{b}\oplus \CC^{c})$, where $V_n$ is the natural representation of  $C_{\frak g}(\frak g_1)\cap\frak g_n$, and $a, b, c\in\ZZ_{>0}$. Hence, for each fixed $t$, $S^t(\bar{\frak n}\cap\frak g_n)\otimes E_n$ is a submodule of an analogous fixed tensor power, and by \refprop{prop11}, $S^t(\bar{\frak n})\otimes E$ is a $C_{\frak g}(\frak g_1)$-module of finite length.
\end{proof}

In the remainder of this section we concentrate on the case $\kk_0=\gg_1$, assuming that $\gg_1$ is non-abelian. In this case $\kk_n=(\gg_1\oplus C_\gg(\gg_1))\cap\gg_n$ is a symmetric subalgebra of $\gg_n$ for $n\geq 2$ and the existing literature on Harish-Chandra modules enables us to prove a stronger version of our main result under slightly different conditions on the compatible parabolic subalgebra $\pp$ and the $\pp$-module $E$. More precisely, let $\pp$ equal $\bigoplus\limits_{\sigma\geq 0}\gg_h^\sigma$ for some real diagonal matrix $h\in\mathfrak t_0$, and $\mm:=C_\gg(h)$. Then $\frak m$ is the direct sum of a reductive in $\frak k_0$ subalgebra $\frak m'$ and an infinite dimensional subalgebra $\frak m''$ isomorphic to $\gl(\infty)$, $\so(\infty)$ or $\sp(\infty)$. Note that ${\frak m}'' \supseteq C_{\frak g} (\frak k_0)$ and that $(\mm_n,\kk_n\cap\mm_n)$ is a symmetric pair for each $n$.

\begin{theo}\label{th43}
For $\gg$ and $\kk$ as above, let the $\pp$-module $E$ satisfy the condition of \refth{th22}(c). In addition, assume that, for some $N\in\ZZ_{\geq 0}$, $E_N$ is a simple finite dimensional $\frak m_{N}$-module such that $A(\frak p_{N}, E_{N})$  is a simple strict $(\frak g_{N}, \frak k_{N})$-module with non-zero bottom layer. Let $v\in A(\pp,E)$ be a non-zero vector in the image of the bottom layer of $A(\pp_{N},E_{N})$ (the existence of $v$ follows from \refth{th22}(c)) and let $X_v$ be a simple quotient of $U(\gg)\cdot v$. Then
\item[(a)] $X_v$ is a strict $(\gg,\kk)$-module;
\item[(b)] if, for all $n$, $E_n$ has finite length as a $(\kk_n\cap\mm_n)$-module, $X_v=\cup_n(X_v)_n$ where each $(X_v)_n$ is a Harish-Chandra $(\gg_n,\kk_n)$-module.
\end{theo}
\begin{proof}
\item[(a)] Let $\pi: U(\gg)\cdot v\to X_v$ be the projection which defines $X_v$, and let $\kappa :A(\frak p_{N}, E_{N})\to A(\frak p, E)$ be the functorially induced map of $(\frak g_{N}, \frak k_{N})$-modules. By our assumptions,  $(\pi\circ\kappa) (v)\neq 0$ and as $A(\pp_N,E_N)$ is simple, $\pi\circ\kappa\neq 0$ is injective. It follows that $\frak g_{N} [A(\frak p_{N}, E_{N}]\supseteq \frak g [X]\cap\frak g_{N}$.  Since $\frak g_{N} [A(\frak p_{N}, E_{N})] = \frak k_{N}$ and is a $(\frak g,\frak k)$-module we conclude that $\frak g [X]\cap\frak g_{N} = \frak k_{N}$.

The inclusion $\frak g [X]\supset \frak k$ implies the following possibilities for $\gg[X]$. If $\gg=\so(\infty), \sp(\infty)$ $\gg[X]$ equals $\kk$ or $\gg$ as $\kk$ is a maximal subalgebra of $\gg$, and if $\gg=\sl(\infty)$ there are four possibilities for $\frak g [X]$: $\gg$, the two opposite parabolic subalgebras ${\frak q}^\pm$ containing $\kk$, and the subalgebra $\kk$. However, in all cases the only possibility compatible with the equality $\gg[X]\cap\gg_N=\kk_N$ is $\gg[X]=\kk$. This proves (a).
\item[(b)] Define $X_n$ as the image of the functorial map of $A(\frak p_n, E_n)$ to $X$.  We have $A(\frak p_n, E_n) = R^s\Gamma_{\frak k_{0},\frak m_{0}} (\text{ind}^{\frak g_{n}}_{\frak p_{n}} E_n)$, $\frak k_n = \frak k_0 +C_{\frak g_{n}} (\frak k_0)$, and $\frak k_n\cap \frak m_n = \frak m_0 + C_{\frak g_{n}} (\frak k_0)$. The comparison principle yields an isomorphism of $(\frak g_n, \frak k_n)$-modules
$$
A(\frak p_n, E_n) \cong R^s\Gamma_{\frak k_{n},\frak k_{n}\cap\frak m_{n}} (\text{ind}^{\frak g_{n}}_{\frak p_{n}} E_n).
$$
Since $(\mm_n,\kk_n\cap\mm_n)$ and $(\frak g_n, \frak k_n)$ are finite dimensional symmetric pairs, any $(\gg_n,\kk_n)$-module (respectively $(\mm_n,\kk_n\cap\mm_n))$-module) of finite length is also of finite type, and hence is a Harish-Chandra module. Moreover, results in \cite[Ch.\romanV]{KV} imply that if $E_n$ has finite length, then $A(\frak p_n, E_n)$ likewise has finite length. Hence $X_n$ itself has finite length, i.e. is a Harish-Chandra module.
\end{proof}

It is easy to construct $(\mm,\kk\cap\mm)$-modules $E$ which satisfy both the assumptions of \refprop{prop42} and \refth{th43}. To satisfy the assumption of \refth{th43}, we can take $E$ to be the union $\cup_n E_n$ of finite dimensional simple $\mm_n$-modules under appropriate inclusions of $\mm_n$-modules $E_n\hookrightarrow E_{n+1}$. For a fixed $N$, we can take $E_{N}$ (for instance $E_N=\CC_{\lambda_{\pp_N}}$, see \refth{thA1} below) so that $A(\pp_{N}, E_{N})$ is simple with non-zero bottom layer. It is also clear that each $E_n$ can be chosen to be a simple submodule of $T^k(V_n^a\oplus(V_n^*)^b\oplus\CC^c)$ for some fixed $a,b,c,k\in\ZZ_{\geq 0}$. Indeed, one can fix $a,b,c,k$ so that the already chosen $\mm_N$-module $E_{N}$ be a submodule of $T^k(V_N^a\oplus(V_N^*)^b\oplus\CC^c)$ and then, for $n\geq N$, recursively choose $E_n$ as a simple submodule of $T^k(V_n^a\oplus(V_n^*)^b\oplus\CC^c)$ for which there is an injection of $\mm_{n-1}$-modules $E_{n-1}\to E_n$. Such a module $E_n$ clearly exists.

\begin{corollary}\label{cor44}
If $\gg=\sl(\infty),\so(\infty),\sp(\infty)$ and $\kk_0=\gg_1$ where $\gg_1$ is not abelian, then $\kk=\kk_0\oplus C_\gg(\kk_0)$ is a primal subalgebra of $\gg$, and moreover there exists a simple strict $(\gg,\kk)$-module $X$ of finite type such that $X=\cup_n X_n$ where $X_n$ are Harish-Chandra $(\gg_n,\kk\cap\gg_n)$-modules.
\end{corollary}

\section{Appendix: The Fernando-Kac subalgebra of a Vogan-Zuckerman module}

Our aim in this appendix is to relate some of the basic literature on applications of cohomological induction with \refsec{TheRootRed} of this paper. More precisely, we recall the definition of a class of Harish-Chandra modules known as the Vogan-Zuckerman modules, \cite{VZ}, and compute the Fernando-Kac subalgebra of a Vogan-Zuckerman module.

Let $\frak g$ be a finite dimensional reductive Lie algebra (over $\CC$), $\frak k$ be a symmetric subalgebra of maximal rank, $\frak t$ be a Cartan subalgeba of $\frak k$ and let $\frak p$ be a $\frak t$-compatible parabolic subalgebra of $\frak g$.  Fix a Levi decomposition $\frak p = \frak m \crplus \frak n$ of $\frak p$ with $\frak t\subseteq\frak m$, and also a $\frak t$-compatible Borel subalgebra $\frak b\subseteq \frak p$. Then $\frak b\cap\frak k$ is a Borel subalgebra of $\frak k$ and $\frak b\cap\frak m$ is a Borel subalgebra of $\frak m$.  Relative to $\frak b$, let $w_\frak g$ be the longest element in the Weyl group of $\frak t$ in $\frak g$; relative to $\frak b\cap\frak m$ let $w_\frak m$ be the longest element in the Weyl group of $\frak t$ in $\frak m$. Finally, let $\lambda_\p := w_\g \circ w^{-1}_\m (\rho_\b) - \rho_\b$.  Note that ${\lambda_\p}_{\mid [\m, \m]} = 0$, so that $\lambda_\p$ defines a one-dimensional $\p$-module $\CC_{\lambda_{\p}}$.

The induced $\gg$-module $\ind^\g_\p \CC_{\lambda_{\p}}$ and the $(\gg,\kk)$-module $A_\p := R^s \Gamma_{\k,\k\cap\m} (\ind^\g_\p \CC_{\lambda_{\p}})$ have the same central character as the trivial $\g$-module. (Here, as usual, $s = \frac{1}{2} \dim (\k/\k\cap\m)$.)  More generally, if $F:= V_\g (\tilde \lambda)$ and $\tilde\lambda := w_\g \circ w^{-1}_\m (\lambda+\rho_\b) - \rho_\b$, then the induced $\gg$-module $\ind^\g_\p (V_\m (\tilde\lambda))$ and the $(\gg,\kk)$-module $A_\p (F) := R^s \Gamma_{\k,\k\cap\m} (\ind^\g_\p (V_\m (\tilde\lambda))$ have the same central character as $F$. We call $A_\p (F)$ the \textit{Vogan-Zuckerman module} attached to the pair $(\p, F)$.  (This definition can be extended to the case rank $\k < \text{rank}\,\, \g$, but we do not consider this generalization here.)

\begin{theo}\label{thA1}
\item[(a)] The bottom layer of $A_\p $ is simple, in particular non-zero.

\item[(b)] $A_\p (F)$ is a simple $(\g,\k)$-module, which is infinite dimensional if $\p$ is {\it proper} in $\g$.
\end{theo}

\begin{proof}

\item[(a)] By \refle{le21}, the bottom layer of $A_\pp$ is isomorphic to $V_\kk(\lambda^\vee_\pp)$. This implies that the bottom layer of $A_\pp$ is simple if non-zero. To ensure that it is indeed non-zero, we need to verify that $\lambda_\pp^\vee$ is dominant with respect to $\kk$. This follows from \cite[Section 3]{VZ}, where it is established that $V_\kk(\lambda_\pp^\vee)$ is a non-zero constituent of the $\kk$-module $\Lambda^\cdot(\kk^\perp)$.

\item[(b)] For the simplicity of $A_\pp (F)$ see Theorem 8.2 on p. 550 in \cite{KV}. When $\pp$ is proper, it is shown in \cite[Section 2]{VZ} that $A_\pp$ has a non-trivial $\kk$-submodule. Since $A_\pp$ has the central character of the trivial $\gg$-module, $\dim A_\pp=\infty$. By using the translation functor one shows that $A_\pp(F)$ is likewise infinite dimensional.
\end{proof}

From now on we assume that $[\g,\g]$ is simple and that $\p$ is proper in $\g$.  We want a formula for the Fernando-Kac subalgebra associated to $A_\p (F)$.  If $\k$ is maximal in $\g$, clearly  $A_\p (F)$ is a strict $(\g, \k)$-module under our assumptions.  If $\k$ is not maximal, then its orthogonal complement $\k^\perp\subset\gg$ is reducible as a $\k$-module: $\k^\perp = \frak r\oplus\bar{\frak r}$, where $\rr$ and $\bar \rr$ are abelian subalgebras of $\gg$, and $\k \crplus \r$ and $\kk\crplus\bar{\frak r}$ are parabolic subalgebras of $\g$. Moreover, there are precisely four subalgebras of $\g$ containing $\k$: $\k, \k \crplus \r, \k \crplus \bar{\frak r}, \g$.

\begin{theo}\label{thA2}  Assume $[\g,\g]$ is simple, $\k$ is not maximal and $\p$ is proper in $\g$.

\item[(a)] $\g [A_\p (F)] = \k \crplus \r$ if $\bar\r \cap\n = 0$.

\item[(b)] $\g [A_\p (F)] = \k \crplus \bar\r$ if $\r\cap\n = 0$.

\item[(c)]  $\g [A_\p (F)] = \k$ if $ \r\cap\n$ and $\bar\r\cap\n$ are both nonzero.
\end{theo}
The proof of \refth{thA2} is based on a lemma relating $\gg[A_\pp]$ with $\Hom_\kk(\Lambda^{\cdot,\cdot}(\r\oplus\bar{\r}),A_\pp)$, where $\Lambda^{\cdot,\cdot}$ stands for bigraded exterior algebra. Set $a:= \dim \bar\r\cap\n$ and $b:= \dim \r\cap \n$. Then, according to the key Proposition 6.19 of \cite{VZ}, $\Hom_\k (\Lambda^{\cdot ,\cdot} (\r\oplus \bar\r), A_\p)$ is concentrated in bidegrees of the form $(a + j, b+ j)$.
\begin{lemma}\label{leA1}

\item[(a)] $\g [A_\p] = \k \crplus \r \Leftrightarrow a = 0$.

\item[(b)]  $\g [A_\p] = \k \crplus \bar\r \Leftrightarrow b = 0$.

\item[(c)] $\g [A_\p] = \k \Leftrightarrow a\neq 0$ and $b\neq 0$.
\end{lemma}
\textbf{Proof of \refle{leA1}.}
\item[(a)] $\g [A_\p] = \k \crplus \r$ if and only if there exists a simple finite dimensional $\k$-module $V$ such that $A_\p$ is isomorphic to the unique irreducible quotient $L(\kk\crplus\rr,V)$ of $\ind^\g_{\k \oplus\r} V$.   But the central character of $A_\p$ is trivial and this constrains $V$ to a finite set: $V$ must be a $\k$-type in $\Lambda^\cdot (\bar\r)$. Hence, $\g [A_\p] = \k \crplus \r$ implies $\Hom_\k (\Lambda^\cdot(\bar\r), A_\p)\neq 0$ which in turn implies $a = 0$.

Conversely, suppose $a = 0$. Let, for some simple finite dimensional $\kk$-module $V$, the $V$-isotypic subspace $A_\pp[V]$ of $A_\pp$ be in the bottom layer of $A_\p$. Theorem 2.5 in \cite{VZ} gives a necessary condition for a simple $\kk$-module $V$ to occur in the restriction of $A_\p$ to $\k$. This condition implies that $\r\cdot A_\p [V] = 0$.  Hence $A_\p\cong L(\k \crplus \r, V)$.

\item[(b)]  Repeat proof of (a) but substitute $\bar\r$ for $\r$.

\item[(c)] Follows from the combination of (a) and (b) and the statement above about $\Hom_\k (\Lambda^{\cdot,\cdot} (\r \oplus\bar\r), A_\p)$.
\qed

\textbf{Proof of \refth{thA2}} First we reduce to the case $F = \CC$, $\lambda = 0$: for any $F$ we have a pair of translation functors $\phi_\lambda$ and $\psi_\lambda$ such that $A_\p (F)\cong \phi_\lambda (A_\p)$ and $A_\p\cong \psi_\lambda (A_\p (F))$ (see \cite[Ch.$\overline{\underline{\mathrm{VII}}}$,Thm.7.237]{KV}). Since $\phi_\lambda (A_\p)$ is a direct summand of $F\otimes A_\p$, we have $\g [A_\p(F)]\supseteq \g [A_\p]$.  Likewise, $\psi_\lambda (A_\p (F))$ is a direct summand of $F^* \otimes A_\p (F)$.  Hence, $\g [A_\p]\supseteq \g [A_\p (F)]$.  Thus, $\g [A_\p (F)] = \g [A_\p]$. \qed

{\bf Example.} Let $\g = sl(n)$ with $n = p + q$, $p > 1$ and $q > 0$, and $\k = s(gl(p)\oplus gl(q))$, the traceless matrices in the subalgebra  $gl(p)\oplus gl(q)$ embedded in the standard fashion in $gl(n)$.  We have $\k = sl(p)\oplus gl(q)$, where $gl(q)$ is embedded as the centralizer of $sl(p)$ in $\g$.  Let $\t\subseteq \k$ be the diagonal matrices; $\t$ is a Cartan subalgebra of $\k$ and of $\g$.  Choose any real nonzero matrix $h\in\t\cap sl(p)$ and let $\p$ be the $\t$-compatible parabolic subalgebra associated to $h\in\t$.  The subalgebra $\k$ is not maximal and we have a triangular decomposition $\g = \r\oplus\k\oplus\bar\r$, where $\r$ and $\bar\r$ are nonzero simple $\k$-submodules of $\g$.  Furthermore, since $h$ has both positive and negative diagonal values, $\p\cap \r\neq 0$ and $\p\cap\bar\r\neq 0$. Therefore, for any simple finite dimensional $\gg$-module $F$, \refth{thA2}(c) implies that $A_{\p}(F)$ is a strict simple $(\g,\k)$-module.

\bibliography{ref,outref,mathsci}

\vskip .20in

\noindent Ivan Penkov

\noindent Jacobs University Bremen \\
\noindent Campus Ring 1\\
\noindent D-28759 Bremen, Germany\\
\noindent email:  i.penkov@jacobs-university.de \\
\vskip .20in

\noindent Gregg Zuckerman\\
\noindent Department of Mathematics\\
\noindent Yale University\\
\noindent 10 Hillhouse Avenue, P.O. Box 208283\\
\noindent New Haven, CT 06520-8283, USA\\
\noindent email:  gregg@math.yale.edu

\end{document}